\documentclass[a4paper,12pt]{article}
\usepackage[T2A]{fontenc}
\usepackage[dvips]{graphicx}
\usepackage{color}
\usepackage{graphicx}
\usepackage{amsmath,amssymb,amsthm}
\usepackage{wrapfig}

\newtheorem{lemma}{Lemma}

\newtheorem{thm}{Theorem}
\newtheorem{cor}{Corollary}[section]
\newtheorem{prop}{Proposition}[section]
\newcounter{def}

\newcommand{\dist}{\mathop{\rm dist}\nolimits}
\newcommand{\cg}{\mathop{\rm CG}\nolimits}

\begin{document}

\title{The strong thirteen spheres problem}
\author{Oleg R. Musin\thanks{Research supported in part by NSF
grant DMS-0807640 and NSA grant MSPF-08G-201.}, Alexey S. Tarasov \thanks{Research supported by program P15 of presidium of RAS and RFBR grant 08-07-00430.}}
\maketitle

\begin{abstract}
The thirteen spheres problem is asking if  13 equal size nonoverlapping
spheres  in three dimensions can touch another sphere of the same size.
This problem was the subject of the famous discussion between Isaac Newton
and David Gregory in 1694. The
problem was solved by Sch\"utte
and van der Waerden only in 1953.

A natural extension of this problem is the strong thirteen
spheres problem (or the Tammes problem for 13 points) which asks to find an arrangement and the maximum
radius of 13 equal size nonoverlapping spheres touching the unit sphere. In the paper we give a solution of this long-standing open problem in geometry.  Our computer-assisted proof  is based on a  enumeration
of the so-called irreducible  graphs.
\end{abstract}

\section{Introduction}

\subsection{The thirteen spheres problem}

The {\it kissing number} $k(n)$ is the highest number of equal
nonoverlapping spheres in ${\Bbb R}^n$ that touch another
sphere of the same size. In three dimensions the kissing number
problem is asking how many white billiard balls can
{\em kiss} (touch) a black ball.

The most symmetrical configuration, 12 balls around
another, is achieved if the 12 balls are placed at positions
corresponding to the vertices of a regular icosahedron concentric
with the central ball. However, these 12 outer balls do not kiss
each other and may all be moved freely. So perhaps if you moved
all of them to one side, a 13th ball would possibly fit in?

This problem was the subject of the famous discussion between Isaac
Newton and David Gregory in 1694 (May 4, 1694; see  \cite{Sz1}
for details of this discussion). Most reports say that Newton
believed the answer was 12 balls, while Gregory
thought that 13 might be possible. However, Casselman \cite{Cas}
found some puzzling features in this story.

This problem is often called the {\it thirteen spheres problem}.
Hoppe \cite{Hop} thought he had solved the problem (1874). But
there was a mistake - an analysis of this mistake was published
by Hales in 1994 \cite{Hales}
(see also \cite{Sz1}).
Finally this problem was solved by Sch\"utte and van der Waerden
in 1953 \cite{SvdW2}. A subsequent two-page sketch of an elegant
proof was given  by Leech \cite{Lee} in 1956. 
Leech's proof was presented in the first edition  of the well-known book
by Aigner and Ziegler \cite{AZ}, the authors removed this chapter from the second edition because a complete proof would have to include so much spherical trigonometry.  

The thirteen spheres problem continues to be of interest, and
new proofs have been published in the last several years by Hsiang
\cite{Hs}, Maehara \cite{Ma, Manew} (this proof is based on Leech's proof),  B\"or\"oczky \cite{Bor},
Anstreicher \cite{Ans}, and Musin \cite{Mus13}.

Note that for $n>3$ the kissing number problem is solved only
for $n=8, 24$ \cite{Lev2,OdS}, and for $n=4$ \cite{Mus4} (see  \cite{PZ} for a beautiful exposition of this problem).

\subsection{The Tammes problem}

If $N$ unit spheres kiss the unit sphere in ${\Bbb R}^n$, then
the set of kissing points
is an arrangement on the central sphere such that the (Euclidean)
distance between any two points is at least 1. So the kissing
number problem can be stated in other way: How many points can
be placed on the surface of ${\Bbb S}^{n-1}$ so that the angular
separation between any two points be at least $60^{\circ}$?

This leads to an important generalization: a finite subset $X$
of ${\Bbb S}^{n-1}$ is called a {\it spherical $\psi$-code} if for
every pair  $(x,y)$ of $X$ with $x\ne y$ 
its  angular distance $\dist(x,y)$ is at least $\psi$.

Let $X$ be a finite subset of ${\Bbb S}^{2}$. Denote
$$\psi(X):=\min\limits_{x,y\in X}{\{\dist(x,y)\}}, \mbox{ where } x\ne y.$$
Then $X$ is a spherical $\psi(X)$-code.

Denote by $d_N$ the largest angular separation $\psi(X)$ with $|X|=N$ that can be attained
 in ${\Bbb S}^{2}$, i.e.
$$
d_N:=\max\limits_{X\subset{\Bbb S}^2}{\{\psi(X)\}}, \, \mbox{ where } \;  |X|=N.
$$
In other words,
{\it how are $N$ congruent, not overlapping circles distributed on the sphere
when their common radius of the circles has to
be as large as possible?}

This question, also known as the problem of the ``inimical dictators'': {\it Where should $N$ dictators build their palaces on a planet so as to be as far away from each other as possible?} The problem 
was first asked by the Dutch botanist Tammes \cite{Tam} (see \cite[Section 1.6: Problem 6]{BMP}), who
was led to this problem by examining the distribution of
openings  on the pollen grains of different flowers.

The Tammes problem is presently solved only for several values
of $N$: for $N=3,4,6,12$ by L. Fejes T\'oth \cite{FeT0}; for
$N=5,7,8,9$ by Sch\"utte
and van der Waerden \cite{SvdW1}; for $N=10,11$ by Danzer \cite{Dan} (for  $N=11$ see also B\"or\"oczky \cite{Bor11});
and for $N=24$ by Robinson \cite{Rob}.

\subsection{The Tammes problem for $N=13$}

The first unsolved case of the Tammes problem is $N=13$, which
is particularly interesting because of its relation to the kissing
problem and the Kepler conjecture \cite{BS,FeT,Sz1}.

Actually this problem is equivalent to {\it the strong thirteen
spheres problem}, which asks to find an arrangement and the maximum
radius of 13 equal size nonoverlapping spheres in ${\Bbb R}^3$
touching the unit sphere.

It is clear that the equality $k(3)=12$ implies $d_{13} < 60^{\circ}$.
B\"or\"oczky and Szab\'o \cite{BS} proved that $d_{13} < 58.7^{\circ}$.
Recently Bachoc and Vallentin \cite{BV} have shown that $d_{13}
< 58.5^{\circ}$.

We note that there is an arrangement of 13 points on
${\Bbb S}^{2}$ such that the distance between any two points of
the arrangement is at least  $57.1367^{\circ}$ (see \cite[Ch. VI, Sec. 4]{FeT}). This arrangement
is shown in Fig.~\ref{fig1}.

\begin{figure}[htb]
\begin{center}
\includegraphics[clip,scale=1]{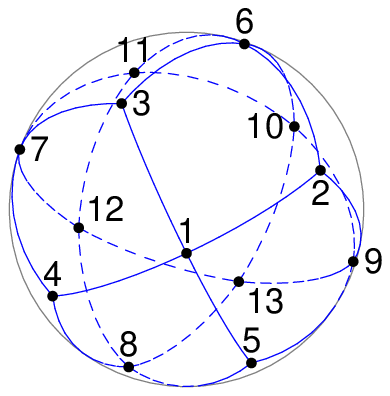}
~~~~
\includegraphics[clip,scale=1]{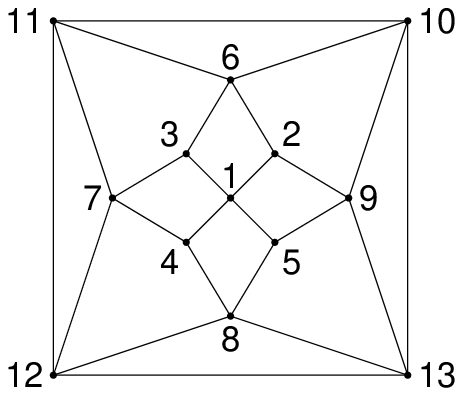}
\end{center}
\caption{An arrangement of 13 points $P_{13}$ and its contact graph $\Gamma_{13}$ with $\psi(P_{13})\approx 57.1367^{\circ}$.}
\label{fig1}
\end{figure}

\medskip

\noindent{\bf Remark.} Denote the constant $\psi(P_{13})$ by $\delta_{13}$. The value $d=\delta_{13}$ can be found analytically. Indeed, we have (see for notations and functions  Fig.~\ref{fig9} and Section 3): $u_0+2u_{13}+u_2=2\pi$, where $u_2=\pi/2, \; a:=u_0=\alpha(d), \; u_{13}=\rho(u_9,d), \; u_9=2\pi-2u_5,\; u_5=\rho(u_2,d)$. This yields:
$$
2\tan\left(\frac{3\pi}{8}-\frac{a}{4}\right)=\frac{1-2\cos{a}}{\cos^2{a}}, \; \; \cos{d}=\frac{\cos{a}}{1-\cos{a}}.
$$
Thus, we have $a_{13}:=\alpha(\delta_{13})\approx 69.4051^{\circ}$ and $\delta_{13}\approx 57.1367^{\circ}$.

\section{Main theorem}

In this paper we present a solution of the Tammes problem for $N=13$.

\begin{thm} The arrangement of 13 points in ${\Bbb S}^2$ which is shown in Fig. 1
is the best possible, the maximal arrangement is unique up to isometry,  and   $d_{13}=\delta_{13}$.
\end{thm}

\subsection{Basic definitions}

\noindent{\bf Contact graphs.} Let $X$ be a finite set in ${\Bbb
S}^2$. The {\it contact graph} $\cg(X)$  is the graph with vertices
in $X$ and edges $(x,y), \, x,y\in X$ such that $\dist(x,y)=\psi(X)$.

\medskip

\noindent{\bf Shift of a single vertex.} Let $X$ be a finite
set in ${\Bbb S}^2$. Let $x\in X$ be a vertex of $\cg(X)$ with
$\deg(x)>0$, i.e. there is $y\in X$ such that $\dist(x,y)=\psi(X)$.
We say that there exists a shift of $x$ if $x$ can be
slightly shifted to $x'$ such that $\dist(x',X\setminus\{x\})>\psi(X)$.

\medskip

\noindent{\bf Danzer's flip.}
Danzer \cite[Sec. 1]{Dan} defined the following flip. Let
$x,y,z$ be  vertices of $\cg(X)$ with $\dist(x,y)=\dist(x,z)=\psi(X)$. We say that $x$ is flipped
over $yz$ if $x$ is replaced by its mirror image $x'$ relative
to the great circle $yz$ (see Fig.~\ref{fig3}). We say that this flip
is {\it Danzer's flip} if $\dist(x',X\setminus\{x,y,z\}) > \psi(X)$.

\medskip

\begin{figure}[h]
\begin{center}
\includegraphics[clip,scale=1]{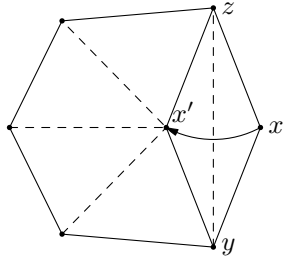}
\end{center}
\caption{Danzer's flip}
\label{fig3}
\end{figure}

\medskip

\noindent{\bf Irreducible graphs.}
We say that the graph $\cg(X)$ is {\it irreducible}\footnote{This
terminology was used by Sch\"utte - van der Waerden \cite{SvdW1,SvdW2},
Fejes T\'oth \cite{FeT}, and Danzer \cite{Dan}.}  (or {\it
jammed}) if there are neither Danzer's flips nor shifts of vertices.

\medskip

\noindent{\bf $P_{13}$ and $\Gamma_{13}$.}
Denote by $P_{13}$ the arrangement of 13 points in Fig. 1. Let
$\Gamma_{13}:=\cg(P_{13})$. It is not hard  to see that the graph
$\Gamma_{13}$ is irreducible.

\medskip
\noindent{\bf Maximal graphs $G_{13}$.}
Let $X$ be a subset of ${\Bbb S}^2$ with $|X|=13$ and $\psi(X)=d_{13}$.
Denote by $G_{13}$ the graph $\cg(X)$.
Actually, this definition does not assume that $G_{13}$ is unique.
We use this designation for some $\cg(X)$ with $\psi(X)=d_{13}$.

\medskip

\noindent{\bf Graphs $\Gamma_{13}^{(i)}$.} Let us define four
planar graphs $\Gamma_{13}^{(i)}$ (see Fig. ~\ref{eliminated}),
where $i=0,1,2,3$, and $\Gamma_{13}^{(0)}:=\Gamma_{13}$. Note
that $\Gamma_{13}^{(i)}, \, i>0$, is obtained from $\Gamma_{13}$
by removing certain edges.

\begin{figure}[h]
\begin{center}
\includegraphics[clip,scale=1]{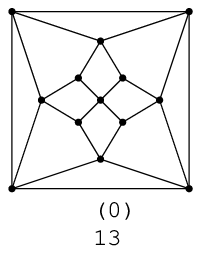}
~~~~
\includegraphics[clip,scale=1]{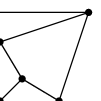}
~~~~
\includegraphics[clip,scale=1]{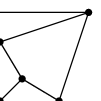}
~~~~
\includegraphics[clip,scale=1]{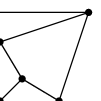}
\end{center}
\caption{Graphs $\Gamma_{13}^{(i)}$ .}
\label{eliminated}
\end{figure}



\subsection{Main lemmas}

\begin{lemma} $G_{13}$ is isomorphic to  $\Gamma_{13}^{(i)}$ with $i=0,1,2,$ or $3$.
\end{lemma}

\begin{lemma}
 $G_{13}$ is isomorphic to $\Gamma_{13}^{(0)}$  and   $d_{13}=\delta_{13} \approx 57.1367^{\circ}$.
\end{lemma}

It is clear that Lemma 2 yields Theorem 1. Now our goal
is to prove these lemmas.

\section{Properties of $G_{13}$}

\subsection{Combinatorial properties of $G_{13}$}

\begin{prop} Let $X$ be a finite set in ${\Bbb S}^2$. Then  $\cg(X)$ is a planar graph.
\end{prop}
\begin{proof} Let $a,b,x,y\in X$ with $\dist(a,b)=\dist(x,y)=\psi(X)$.
Then the shortest arcs ${ab}$ and ${xy}$ don't intersect. Otherwise,
the length of at least one of the arcs $ax, ay, bx, by$ has to be
less than $\psi(X)$. This yields the planarity of $\cg(X)$.
\end{proof}

The following three  propositions are proved in \cite{Dan} (also see \cite[Chap. VI]{FeT}, \cite{BS,BS14}).

\begin{prop} Let $X$ be a subset of ${\Bbb S}^2$ with $|X|=N$ and $\psi(X)=d_N$. Then for  $N>6$ the graph $\cg(X)$ is irreducible.
\end{prop}

\begin{prop} Let $X\subset {\Bbb S}^2$. If the graph $\cg(X)$ is  irreducible, then degrees of its vertices can take only the values $0$ (isolated vertices), $3$, $4$, or  $5$.
\end{prop}

\begin{prop} Let $X\subset {\Bbb S}^2$ with $|X|=N$. If the graph $\cg(X)$ is  irreducible, then its faces are polygons with at most $\lfloor2\pi/d_N\rfloor$ vertices.
\end{prop}

B\"or\"oczky and Szab\'o {\cite[Lemma 8 and Lemma 9(iii)]{BS}} considered isolated vertices in  irreducible graphs  with 13 vertices.
\begin{prop}  Let $X\subset {\Bbb S}^2$ with $|X|=13$. Let the graph $\cg(X)$ be irreducible.
 If  $\cg(X)$ contains an isolated vertex, then it lies in the interior of a hexagon of $\cg(X)$ and this hexagon cannot contain other vertices of $\cg(X)$.
\end{prop}

Combining these propositions, we obtain the following combinatorial properties of $G_{13}$.
\begin{cor}
\begin{enumerate}
\item $G_{13}$ is a planar graph;
\item Any vertex of $G_{13}$ is of degree $0,3,4,$ or $5$;
\item Any face of $G_{13}$ is a polygon with $3,4,5$ or $6$ vertices;
\item If  $G_{13}$ contains an isolated vertex $v$, then $v$ lies in a hexagonal face. Moreover, a hexagonal face of  $G_{13}$ cannot contain two or more isolated vertices.
\label{cor1}
\end{enumerate}

\end{cor}

\subsection{Geometric properties of $G_{13}$}

Let $X\subset {\Bbb S}^2$ with $|X|=13$. Let the graph $\cg(X)$
be irreducible. Note that all faces of $\cg(X)$ are convex polygons. (Otherwise, a ``concave'' vertex of a polygon $P$ can be shifted to the interior of $P$.) Then the faces of the graph $\cg(X)$ in ${\Bbb S}^2$ are
regular triangles, rhombi, convex equilateral pentagons, and
convex equilateral hexagons. Polygons with more than six vertices
cannot occur. Note that the triangles, rhombi, or pentagons
of $\cg(X)$ cannot contain isolated vertices in their interiors.
The lengths of all edges of $\cg(X)$ equal $\psi(X)$.

Consider as parameters (variables) of $\cg(X)$ in ${\Bbb S}^2$
the set of all angles $u_i$ of its faces and $d:=\psi(X)$. Clearly, the graph $G=\cg(X)$,
$d$, and the set $\{u_i\}$ uniquely (up to isometry) determine
embedding $X\setminus\{$isolated vertices$\}$ in ${\Bbb S}^2$.

We obviously have the following constraints for these parameters.

\begin{prop}

\begin{enumerate}
\item $u_i <\pi$ for all $u_i$;
\item $u_i \ge \alpha(\psi(X))$ for all $u_i$, where
$$ \alpha(d): =\cos^{-1}\left( \frac{ \cos{d}}{1 + \cos{d}}\right)$$
is the angle of a regular triangle in ${\Bbb S}^2$
with sides of length $d$;
\item $\sum_{k \in I(v)} u_k = 2 \pi$ for all vertices $v$ of
$G$, where $I(v)$ is the set of
subscripts of angles that are adjacent  to $v$;
\end{enumerate}
\end{prop}

Let $F$ be a face of $G$. Then $F$ is a polygon with $m$ vertices, where $m=3,4,5$, or $6$. Consider all possible cases.

\medskip

\noindent{\bf 1. $m=3$: triangle.} In this case, $F$ is a regular
triangle.
\begin{prop} Let $F$ be a triangular face of $G_{13}$ with angles $u_1,u_2,u_3$. Then
 $u_{1}=u_{2}=u_{3}=\alpha_{13}:=\alpha(d_{13})$.
\end{prop}

\medskip

\noindent{\bf 2. $m=4$: quadrilateral.} In this case, $F=A_1A_2A_3A_4$ is a
rhombus. Then we have
 $u_{1}=u_{3}$ and $u_{2}=u_{4}$. Using
the spherical Pythagorean theorem, one can show that
$$
\cot{\frac{u_{1}}{2}}\,\cot{\frac{u_{2}}{2}} =\cos d.
$$
Then
$$
u_2=\rho(u_1,d):=2\cot^{-1}(\tan{(u_{1}/2)}\cos d).
$$

Since $u_2\ge\alpha(d)$, we have $u_1=\rho(u_2,d)\le\rho(\alpha(d),d)=2\alpha(d)$ (Fig.~\ref{rect}).
\begin{prop} Let $F$ be a quadrilateral of $G_{13}$ with
angles $u_1,u_2,u_3,u_4$. Then $u_3=u_1,\, u_4=u_2,\, u_2=\rho(u_1,d_{13}), \,  u_1=\rho(u_2,d_{13}),$ and
$\alpha_{13}\le u_i\le 2\alpha_{13}\, $ for all $\, i=1,2,3,4$.
\end{prop}

\begin{figure}[htb]
\begin{center}
\includegraphics[clip,scale=1]{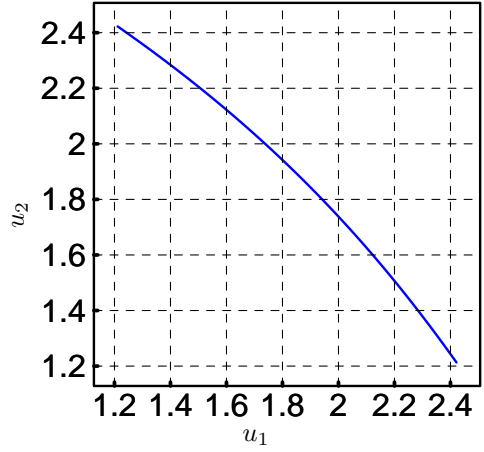}
\end{center}
\caption{The graph of the function $u_2=\rho(u_1,d)$,  where
$d=57.1367^\circ$.}
\label{rect}
\end{figure}



\medskip

\noindent{\bf 3. $m=5$: pentagon.} In this case, $F$ is a convex
equilateral pentagon $A_1A_2A_3A_4A_5$. Let
$u_{1},u_{2},u_{3},u_{4},u_{5}$ be its angles. Then $F$ is uniquely
determined by $d$ and any pair of these angles, for instance, by $(u_1,u_2)$ (Fig.~\ref{penta2}).

\begin{figure}[htb]
\begin{center}
\includegraphics[clip,scale=1]{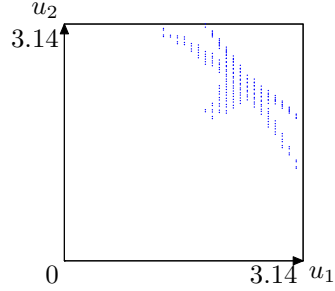}
\end{center}
\caption{The set of admissible pairs $(u_{1},u_{2})$ for a pentagon with $d=57.1367^\circ$.}
\label{penta2}
\end{figure}

It is not hard for given parameters
$x=u_1, y=u_2$ and $d$ to find $u_3,u_4,u_5$ as functions of $x,y,d$, i.e. $u_i=f_i(x,y,d)$, where $i=3,4,5$.
Let $f_1(x,y,d)=x$ and $f_2(x,y,d)=y$. Then we have $u_i=f_i(x,y,d)$ for all $i=1,\ldots,5$. We have that all $f_i(x,y,d)\ge\alpha(d)$.

Denote by $A_i'$ the image of $A_i$ after Danzer's flip.
Let $\xi_i(x,y,d)$ denote the minimum distance between $A_i'$
and $A_j$, where $j\ne i$. If $F$ is a face of $\cg(X)$ and
$\cg(X)$ is irreducible, then $F$ does not admit Danzer's flips.
Therefore, $\xi_i(x,y,d)<d$ for all $i$. Thus we have the following
proposition.
\begin{prop} Let $F$ be a pentagonal face of $G_{13}$ with
angles $u_1,\ldots,u_5$. Then $f_i(u_1,u_2,d_{13})\ge\alpha_{13}$ and
$\xi_i(u_1,u_2,d_{13})<d_{13}$ for all $i=1,\ldots,5$.
\end{prop}


\begin{figure}[htb]
\begin{center}
\includegraphics[clip,scale=1]{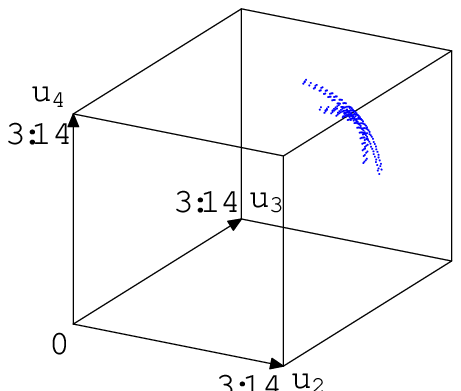}
~~~~
\includegraphics[clip,scale=1]{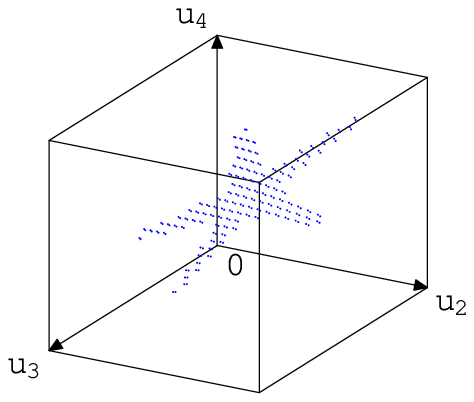}
\end{center}
\caption{Admissible angles $(u_{1},u_{2},u_{3},u_{4},u_{5})$ of a pentagon
 projected into $(u_{2},u_{3},u_{4})$}
\label{penta}
\end{figure}

\medskip

\noindent{\bf 4. $m=6$: hexagon.} In this case, $F=A_1A_2A_3A_4A_5A_6$ is a convex
equilateral hexagon with angles $u_1,\ldots,u_6$. Clearly,
$F$ is uniquely defined by any three angles and $d$.

Let $u_i=g_i(u_1,u_2,u_3,d)$ for  $i=4,5,6$. Let $g_i(u_1,u_2,u_3,d)=u_i$ for $i=1,2,3$. Then we have $u_i=g_i(u_1,u_2,u_3,d)$ for all $i=1,\ldots,6$.

In fact, for the case $m=6$ we have two subcases: (a) $F$ has
no isolated vertices, and (b) $F$ has an isolated vertex.

It is easy to see that for case 4(a) there exists an analog of Proposition 3.9.
Let $\zeta_i(u_1,u_2,u_3,d)$ denote the minimum distance between $A_i'$ and $A_j$, where $j\ne i$.
\begin{prop} Let $F$ be a hexagonal face of $G_{13}$ with
angles $u_1,\ldots,u_6$. Suppose the face $F$ has no isolated vertices in its interior. Then $g_i(u_1,u_2,u_3,d_{13})\ge\alpha_{13}$ and
$\zeta_i(u_1,u_2,u_3,d_{13})<d_{13}$ for all $i=1,\ldots,6$ (Fig.~\ref{emptyhex}).
\end{prop}

\begin{figure}[htb]
\begin{center}
\includegraphics[clip,scale=1]{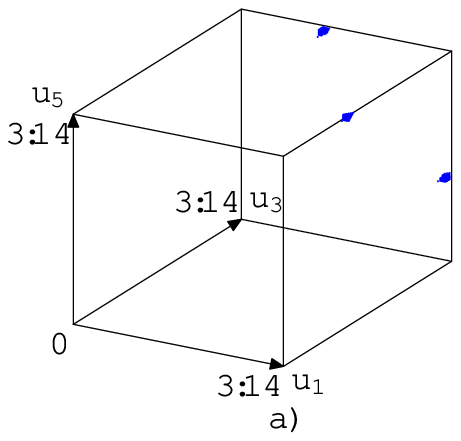}
~~~~
\includegraphics[clip,scale=1]{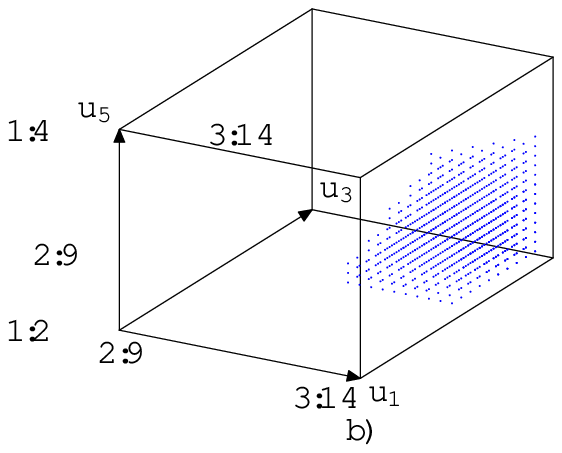}
\end{center}
\caption{(a) The set of admissible triplets for $(u_{1},u_{2},u_{3})$ for empty hexagon with $d=57.1367^{\circ}$. (b) A component with zoom.}
\label{emptyhex}
\end{figure}

Now consider case 4(b). Denote by $\Pi$ the set of all points
$p$ in the interior of $F$ such that there is a pair $(i,j),\,
1\le i,j\le 6, \, i\ne j,$ with
$\dist(p,A_i)=\dist(p,A_j)=d$. Clearly,  $|\Pi|\le 18$.

Let  $p\in\Pi$ be defined by a pair $(i,j)$. Denote by $K(p)$
the set of all $k=1,\ldots,6$ such that $k\ne i$ and $k\ne j$.
Let
$$
\lambda(u_1,u_2,u_3,d)=\tilde\lambda(F):=\max\limits_{p\in\Pi}{\min\limits_{i\in
K(p)}\{\dist(p,A_i)\}}.
$$
Since $F$ contains an isolated vertex, we have $\tilde\lambda(F)\ge d$.
\begin{prop} Let $F$ be a hexagonal face of $G_{13}$ with
angles $u_1,\ldots,u_6$. Suppose the face $F$ has an isolated vertex in
its interior.

Then $g_i(u_1,u_2,u_3,d_{13})\ge\alpha_{13}$  for
all $i=1,\ldots,6$ and $\lambda(u_1,u_2,u_3,d_{13})\ge d_{13}$.
\end{prop}

\section{Proof of Lemma 1}

Here we give a sketch of our computer proof. For more details
see http://dcs.isa.ru/taras/tammes13/~.

The proof consists of two parts:\\
(I) Create the list $L_{13}$ of all graphs with 13 vertices that
satisfy Corollary 3.1;\\
(II) Using linear approximations and linear programming remove
from the list $L_{13}$ all graphs that do not
satisfy the geometric properties of $G_{13}$ (see Propositions
3.6-3.11).

\medskip

(I). To create $L_{13}$ we use the program {\it plantri}
(see \cite{PLA2}).\footnote{The authors of this program are Gunnar
Brinkmann and Brendan McKay.} This program  is the  isomorph-free
generator  of planar graphs, including triangulations, quadrangulations, and
convex polytopes.
(The paper \cite{PLA1} describes plantri's principles of operation,
the basis for its efficiency,
and recursive algorithms behind many of its capabilities.)

The program plantri generates 94,754,965 graphs in $L_{13}$,
i.e. graphs that satisfy Corollary 3.1. Namely, $L_{13}$
contains 30,829,972 graphs with triangular and quadrilateral
faces; 49,665,852 with at least one pentagonal face  and with
triangular and quadrilaterals;
13,489,261 with at least one hexagonal face which do not contain
 isolated vertices;
769,375\footnote{Perhaps contains isomorphic graphs.} graphs with one isolated vertex,
$505$\footnotemark[3]  with two isolated vertices, and no
graphs with three or more isolated vertices.

\medskip

(II). Let us consider a graph $G$ from $L_{13}$. We start from
the level of approximation $\ell=1$.  Now using
Propositions 3.6-3.11 we write  linear equalities and inequalities
for the parameters (angles) $\{u_i\}$ of this graph.

For $\ell=1$ we use the following linear equalities and
inequalities:\\
(i) 13 linear equalities  $\sum_{k \in I(v)} u_k = 2 \pi$ in
Proposition 3.6(3);\\
(ii) Since  $57.1367^\circ=0.9972\le d_{13}<1.021=58.5^\circ$,
we have
$ 1.2113\le\alpha_{13}<  1.2205$;\\
(iii) For a quadrilateral
from Proposition 3.8 we have equalities $u_3=u_1, \, u_4=u_2$,
and inequalities $\alpha_{13}\le u_i\le 2\alpha_{13}, \, i=1,2$;\\
(iv) For a quadrilateral, (ii) and $u_2=\rho(u_1,d_{13})$
 yield $3.6339\le u_1+u_2\le 3.779657$;\\
(v) Let $F$ be  a pentagonal face. Consider all vectors $U_5:=\{(u_1,\ldots,u_5)\}$
that satisfy Proposition 3.9 (see Fig.~\ref{penta}). We use
a convex polytope $P_5$ in ${\Bbb R}^5$ which contains $U_5$.
Actually, $P_5$ is defined by certain linear inequalities. For
instance,
$2.96 \le u_1 + u_2 - 0.63 u_4 \le 3.26$,
 $ u_1+ u_3 + 1.8 u_2 \le 9.05$, etc; \\
(vi) For a hexagonal face $F$ that contains no isolated
vertices, using
Proposition 3.10, we find a set of three polytopes
$P_6^k$, $U_6 \subset \cup_{k=1}^3 P_6^k$  which are defined
by the inequalities
  $1.2 \le u_k, u_{k+3}
\le 1.34$ and $2.9 \le u_{k+1},u_{k+2},u_{k+4},u_{k+5}$;\\
(vii) For a hexagonal face with an isolated vertex, Proposition
3.11 yields $\sum_{i=1}^6 u_i \ge 15.936$.

\medskip

Using this set of linear inequalities,
we find minimal and maximal value of each variable by linear
programming. This gives
us a convex region in the space of possible solutions that contains
all possible solutions
for given graph (if they exist).
If the region becomes empty, this means that we can eliminate
the graph considered.
This step ``kills'' almost all graphs.
After this step, there remain   $2013$ graphs without hexagons,
$40910$ graphs with hexagons
and without isolated vertices, $9073$ graphs
with one isolated vertex, and $272$ graphs with two isolated
vertices.

For $\ell=2$ we use the following idea.
This region is smaller than the original region, so we can adjust
linear estimates for nonlinear equalities and inequalities.
For quadrilaterals we adjust inequalities using (iv).
For pentagons we are using an additional set of inequalities. Namely,
using functions $f_3(u_1,u_5,d)$, $f_3(u_2,u_4,d)$, and bounds for $u_1,u_2,u_4,u_5,d$
can be obtained minimal and maximal linear bounds for $u_3$.

Repeating this procedure , we obtain a chain of nested convex 
regions, which contain all possible solutions. This chain converges to empty or non-empty region.
If this result is empty, the graph is eliminated.
After this step, only $260$ graphs remain in the main group,
$9991$ graphs remain
in the second group, $126$ graphs remain in the third group,
and no graphs remain in the fourth group.

For the level of approximation $\ell=3$, we split the
region into two smaller regions
and repeat the same procedure as for $\ell=2$ independently.
For graphs with
empty hexagons, we make a specific split by taking different values
of $k$ from item (vi) (see above).

Repeating the splitting procedure, we ``kill'' all graphs except $\Gamma_{13}^{(i)}$.

This result gives us two surprises. We expected that subgraphs
were to remain, because they can be infinitesimally close to $\Gamma_{13}$,
and so they cannot be eliminated by computer program.
But we didn't expect that all other graphs would be killed.
Also manually, we found two subgraphs which could be contact graphs:
$\Gamma_{13}^{(1)}$ and $\Gamma_{13}^{(2)}$. But we missed the graph
$\Gamma_{13}^{(3)}$ with one isolated vertex, which was found by computer
program.

\begin{figure}[htb]
\begin{center}
\includegraphics[clip,scale=1]{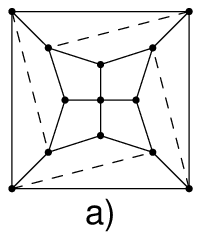}
~~~~
\includegraphics[clip,scale=1]{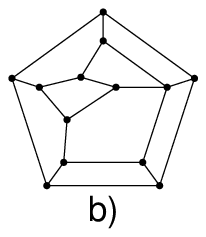}
~~~~
\includegraphics[clip,scale=1]{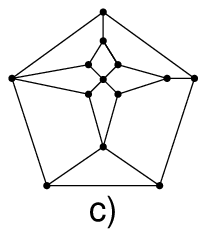}
~~~~
\includegraphics[clip,scale=1]{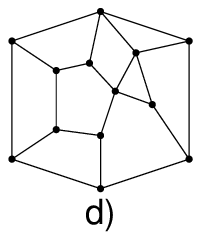}
~~~~
\includegraphics[clip,scale=1]{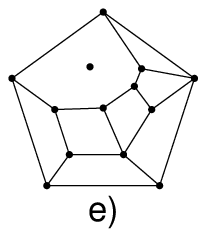}
~~~~
\includegraphics[clip,scale=1]{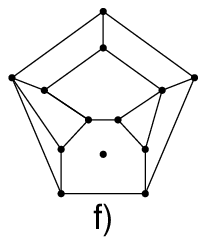}
\end{center}
\caption{Strongest eliminated graphs}
\label{fig6}
\end{figure}

{\noindent \bf Remark}. In Fig.~\ref{fig6} are presented
examples of
graphs which are not isomorphic to $\Gamma_{13}^{(i)}$ and have
been eliminated only after many iterations. The most ``surviving''
graph is $a)$.
This graph is also a subgraph of $\Gamma_{13}^{(0)}$. After eliminating
four edges, the graph contains four pentagons. The reason why
it was eliminated because there are angles $u_i$ which are slightly bigger
than $\pi$, so that the pentagons are not convex. Therefore, this graph is not irreducible.
Other most surviving  graphs were ``strong'' because they have
several pentagons and hexagons. Note that here we use weak bounds for
pentagons  and hexagons given by (v),(vi),(vii).
Our elimination procedure works very fast when we have sufficiently
many triangles and quadrilaterals, and it works worse (slowly) when we have
several pentagons and hexagons.

\section{Proof of Lemma 2}

\begin{proof} This proof is based on geometric properties of $G_{13}$. In 
Section 4 we substitute all nonlinear equations by certain linear inequalities. Note that a statement $d_{13}\approx\delta_{13}$ is a by-product of this approximation. Here we prove that $d_{13}=\delta_{13}$ based on original equations. 
	
Lemma 1 says that $G_{13}=\Gamma_{13}^{(i)}$, where $i=0,1,2$, or 3. We are going to prove that if  $\cg(X)=\Gamma_{13}^{(i)}$ with $i>0$, then $\psi(X)<\delta_{13}=\psi(P_{13})$.

\begin{figure}[h]
\begin{center}
\includegraphics[clip,scale=1]{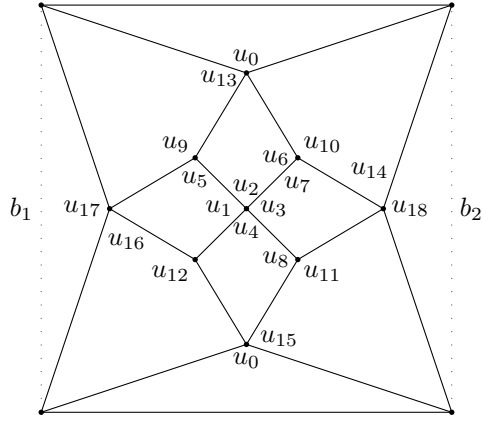}
\end{center}

\caption{Angles of $\Gamma_{13}^{(2)}$}
\label{fig9}
\end{figure}


\noindent{\bf 5.0. Angles of $\Gamma_{13}^{(2)}$.} Let $u_0:=\alpha(d)$. For $G_{13}=\Gamma_{13}^{(2)}$ we have (see Fig.~\ref{fig9}):

\begin{tabular}{ccc}

$u_5 =\rho(u_1,d)$ & 
$u_6 =\rho(u_2,d)$ &
$u_9=2 \pi - u_5 -u_6 $ \\
$u_{13} =\rho(u_{9},d)$ & 
$u_{14} =2\pi-u_0-u_{13}-u_2$ & 
$u_{10}=\rho(u_{14}) $ \\
$u_7 =2\pi -u_6 -u_{10}$ & 
$u_3 =\rho(u_7,d)$ &
$u_4=2 \pi - u_1 -u_2 -u_3$ \\ 
$u_8 =\rho(u_4,d)$ & 
$u_{11}=2 \pi - u_7 -u_8$ & 
$u_{12}=2 \pi - u_8 -u_5 $\\
$u_{15} =\rho(u_{11},d)$ &
$u_{16} =\rho(u_{12},d)$ & 
\end{tabular}

\medskip

Therefore, for $3\le i\le 16$  the value $u_i$  are functions in the variables $u_1,u_2,d$. Since we have also an additional  equation  for the vertex $v_8$ (see  Fig.~\ref{fig1} and Fig.~\ref{fig9}): 
$$u_0+u_{15}+u_4+u_{16}=2\pi,$$ 
the value $d$ is a function in $u_1,u_2$, as well as $u_2$ is a function in the variables $u_1$ and $d$. 
Thus, all $u_i$ and $d$ are functions in $u_1, u_2$ or in $u_1, d$. 

Now we consider three cases $G_{13}=\Gamma_{13}^{(i)}$, where $i=1,2,3$. 

\medskip

\noindent{\bf 5.1. The case $G_{13}=\Gamma_{13}^{(1)}$.} In this case $u_{17}=u_0$. Then for the vertex $v_7$  we have the equation: 
$$u_1+u_{13}+u_0+u_{16}=2\pi.$$ 
From this it follows that $u_1$ and therefore  all $u_i$  are functions in $d$. Note that $$ u_{18}=2\pi-u_{14}-u_3-u_{15}.$$
Thus, $u_{18}$ is a function in $d$ (see Fig.~\ref{u18d}).

\begin{figure}[h]
\begin{center}
\includegraphics[clip,scale=.65]{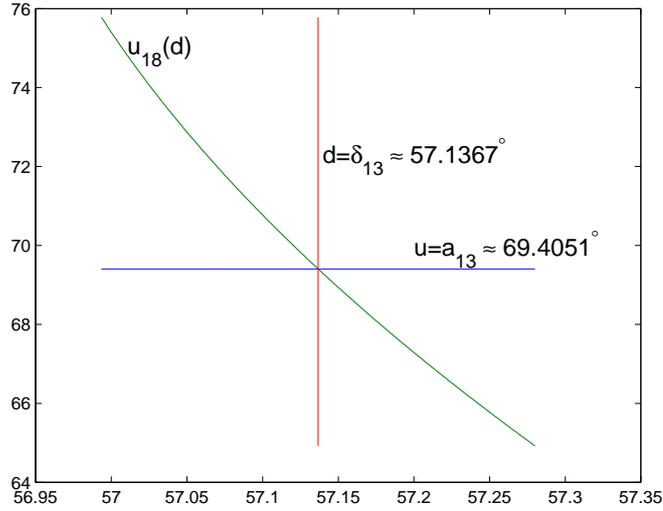}
\end{center}
\caption{The graph of the function $u_{18}(d)$}
\label{u18d}
\end{figure}

If $G_{13}=\Gamma_{13}^{(1)}$, then $u_{18}>u_0\ge a_{13}$. Since the function $u_{18}(d)$ is monotone decreasing, we have $u_{18}(d) > a_{13}$ only if  
$d<\delta_{13}$. 
Thus, $G_{13}\ne\Gamma_{13}^{(1)}$.  

\medskip

\begin{figure}[htb]
\begin{center}
\includegraphics[clip,scale=0.55]{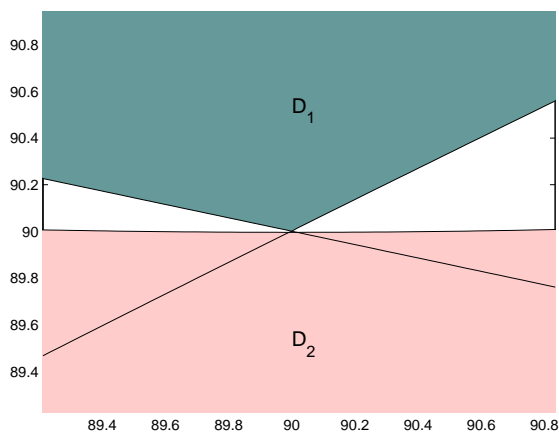}
\end{center}
\caption{$D_1$ and $D_2$}
\label{D12}
\end{figure}


\noindent{\bf 5.2. The case $G_{13}=\Gamma_{13}^{(2)}$.}   It is already shown that $d$ and all $u_i$  are functions in $u_1, u_2$. Let 
$$ 
D_1:=\{(u_1,u_2):  u_{17}\ge u_0, \; u_{18}\ge u_0\} \; \mbox{ and } \: D_2:=\{(u_1,u_2):  u_{0}=\alpha(d)\ge a_{13}\}.
$$

We can see from  Fig.~\ref{D12} that the intersection  $I:=D_1\cap D_2\subset{\Bbb R}^2$  consists of one point with $u_1=u_2=90^\circ.$   It is not hard to prove this fact. Indeed, conversely, $d_{13}>\delta_{13}$ and there is a point $(u_1,u_2)$ on the boundary of $I$ such that $u_{17}=u_0$ or $u_{18}=u_0$.  
Therefore, we have the same case as in 5.1, a contradiction. 
Thus, $G_{13}\ne\Gamma_{13}^{(2)}$. 

\medskip

\noindent{\bf 5.3. The case $G_{13}=\Gamma_{13}^{(3)}$.} This case can be considered by the same method as the case $G_{13}=\Gamma_{13}^{(2)}$. Actually, for given $u_1,u_2,d$ all angles  $u_i, 3\le i\le 16, i\ne 15$ can be found by the same formulas as in 5.0. On the other hand, 
$$
u_{15}=2\pi-u_4-u_{16}-u_0.
$$
Then all $u_i$ are functions in the variables $u_1,u_2,$ and $d$. Since $u_{17}=u_0$ (or equivalently $b_1=d$), we have the equation: 
$$u_1+u_{13}+u_0+u_{16}=2\pi.$$ 
It yields that all $u_i$ depend on two parameters.

\begin{figure}[h]
\begin{center}
\includegraphics[clip,scale=1]{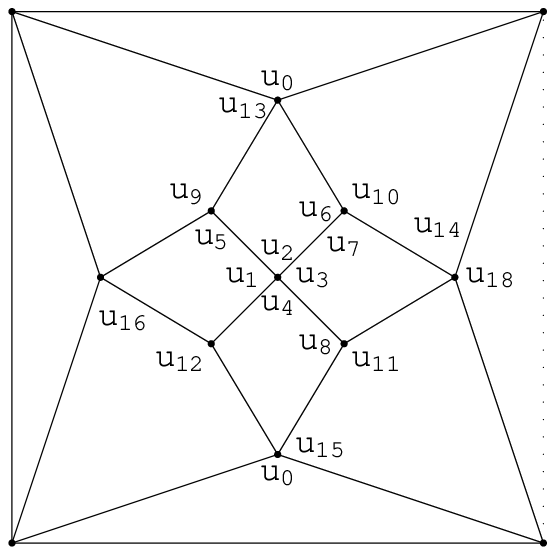}
~~~
\includegraphics[clip,scale=1]{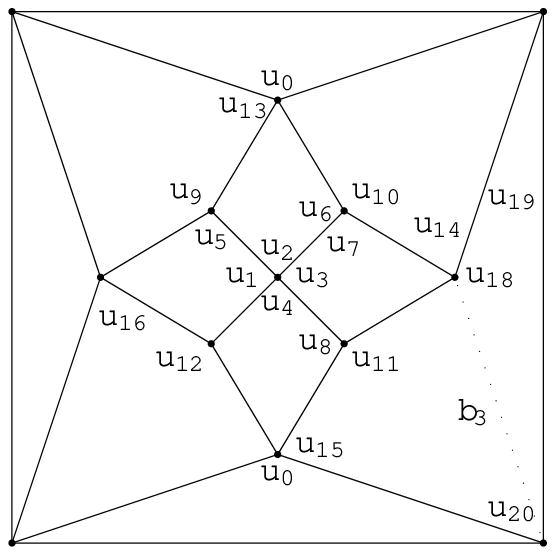}
\end{center}
\caption{Two subcases for the case $G_{13}=\Gamma_{13}^{(3)}$.}
\label{two_cases}
\end{figure}

The vertex $v_{13}$ is isolated. In fact, we can shift this point in such a way that at least two edges $v_{13}v_k$, where $k=8,9,10,12$, have lengths $d$. 
Then for two other edges  we have inequalities: $\dist(v_{13},v_i)\ge d$ and $\dist(v_{13},v_j)\ge d$. 

Arguing as in 5.2, we can show that there are parameters $u_1, u_2$ such that  $u_0>a_{13}$ and at least one of the inequalities $\dist(v_{13},v_k)\ge d, \; k=i, j,$   becomes equality.  
It is not hard to see that there are exactly two geometrically nonequivalent cases with exactly one edge $v_{13}v_k, k=8,9,10,$ or $12$, such that $\dist(v_{13}v_k)>d$. These cases are shown in Fig.~\ref{two_cases}.

Actually, the first subcase is case  5.1. For the second subcase consider the pentagon $F:=v_5v_8v_{12}v_{13}v_{10}$. All angles of $F$ can be found as functions in $u_1,d$. Since $d$ and any two angles of $F$ define all other angles we can use one of these equations to find $u_1$ as a function in $d$. Then $u_{19}$ (see Fig.~\ref{two_cases}) is a function in $d$. In fact, the graph of the function $u_{19}(d)$ is very similar to the graph $u_{18}(d)$ in Fig.~\ref{u18d}, and $u_{19}(d)$ is a monotone decreasing function.   Thus, $u_{19}(d)$ cannot be greater than $a_{13}$, and $G_{13}\ne\Gamma_{13}^{(3)}$. 

\medskip 

We see that if $\cg(X)=G_{13}$, then $\cg(X)$ is isomorphic to $\Gamma_{13}$. Moreover,
 $X$ is uniquely defined up to isometry and $\psi(X)=\delta_{13}\approx
57.1367^\circ$. This completes the proof.
\end{proof}

\medskip

\noindent{\bf Acknowledgment.} We  wish to thank Robert Connelly, Alexey Glazyrin, Nikita Netsvetaev,
and G\"unter Ziegler for helpful discussions and  comments.

\medskip

\medskip

\medskip

\medskip

\medskip

\medskip

O. R. Musin, Department of Mathematics, University of Texas at
Brownsville, 80 Fort Brown, Brownsville, TX, 78520, USA.

 {\it E-mail address:} oleg.musin@utb.edu

\medskip

A. S. Tarasov, Institute for System Analysis, Russian Academy
of Science, 9 Pr. 60-letiya Oktyabrya, Moscow, Russia.

{\it E-mail address:} tarasov.alexey@gmail.com

\end{document}